\documentclass[12pt]{amsart}

\usepackage{amscd,amssymb,amsfonts,amsmath,amstext,amsthm}
\usepackage[all]{xy}

\newtheorem{thm}{Theorem}[section]
\newtheorem{defn}[thm]{Definition}

\newtheorem{cor}[thm]{Corollary}
\newtheorem{prop}[thm]{Proposition}
\newtheorem{lemma}[thm]{Lemma}

\title[A trace-like invariant for represen\textbf{t}ations of Hopf algebras]{A trace-like invariant for represen\textbf{t}ations of Hopf algebras}
\author{Andrea Jedwab}
\address{Claremont McKenna College, Claremont, CA 91711}
\email{ajedwab@cmc.edu}

\begin{document}

\maketitle

%-----------------Introduction---------------------------------

\section{Introduction}
 
In this paper we introduce a trace-like invariant for the irreducible representations of a finite dimensional complex Hopf algebra $H.$  We do so by considering the trace of the map $S_V$ induced by the antipode $S$ on the endomorphisms $End(V)$ of a self-dual module $V$. We also compute the values of this trace for the representations of two non-semisimple Hopf algebras: $u_q(sl_2)$ and $D(H_n(q)),$ the Drinfeld double of the Taft algebra. 

The trace-like invariant $\mu$ that we consider here has some properties in common with the Frobenius-Schur indicator, which has been studied for example in \cite{GM}, \cite{KMM} and \cite{MaN}. The indicator  has been a very useful tool in the study of Hopf algebras and their representations and we expect $\mu$ to be a useful invariant as well. When $H$ is semisimple, the Frobenius-Schur indicator $\nu(V)$ of an irreducible $H$-representation $V$ is an invariant of the tensor category of representations of $H.$ When $V$ is self-dual it was shown in \cite{LM} that, for $S_V$ as above,  $\displaystyle \nu(V)=\frac{Tr(S_V)}{dim(V)}$ and consequently $Tr(S_V)=\pm dim(V)\in \mathbb{Z}.$  We define $\mu(V)$ to be $Tr(S_V)$ for each irreducible representation of any finite dimensional Hopf algebra. %over $\mathbb{C}.$ 
We will see that in the case when $H$ is not semisimple $Tr(S_V)$ might no longer be an integer, although it is still an algebraic integer which depends on the order of $S$ in $H$. We note that when $H$ is not semisimple but is pivotal, \cite{NS} define a generalized indicator which has integer values.  Thus the above formula for $\nu(V)$ does not extend to the more general setting of pivotal Hopf algebras.

Here we determine the trace invariant for the irreducible representations of the Drinfeld double of the Taft algebras $D(H_n(q^2)).$ We do this by considering the projection of $D(H_n(q^{2}))$ onto $u_q(sl_2).$ We will determine the invariant for the  $D(H_n(q^{2}))$-modules by identifying them with $u_q(sl_2)$-modules and computing the invariant for the $u_q(sl_2)$-modules. As a consequence we see that for each simple $u_q(sl_2)$-module $V,\ \displaystyle  \mu(V)=(-1)^{dim(V)+1} dim_q(V)$ where $dim_q(V)$ is the quantum dimension of $V.$

%-------------Section-------------------------------------------

\section{Definition of the invariant $\mu$}

Let $H$ be a finite dimensional Hopf algebra and $V$ a simple self-dual $H$-module with corresponding irreducible representation $\varphi :H\to End(V).$ Since $H$ is finite dimensional, the antipode $S:H\to H$ is bijective. $S$ induces an antialgebra morphism $S_V:End(V)\to End(V)$ in the following way: 

\begin{lemma}
Let $I=ker(\varphi: H\to End(V))$ and $I_*=ker(\varphi_*:H\to End(V^*)),$ where $V^*$ is the dual of $V$. Then $\overline{S}(I_*)=I.$
In particular, if  $V$ is a self-dual simple module, $S$ induces a map on $H/I.$ %\simeq \displaystyle M_{n}(\mathbb{C}). 
\end{lemma}

\begin{proof}
Let $h\in I_*.$ Then $h$ acts as $0$ on $V^*$ and $(h\cdot f)(v)=0,\ \forall f\in V^*,\ v\in V.$ So by definition of the action of $H$ on the dual, for each $v\in V,\ f(S(h)\cdot v)=0\ \forall f\in V^*$ and thus $S(h)\cdot v=0.$ We conclude that $S(I_*)\subset I.$

Now let $h\in I.$ Consider the inverse $S^{-1}$ of the antipode $S$. Then $(S^{-1}(h)\cdot f)(v)=f(h\cdot v)=0,\ \forall v\in V,\ f\in V^*.$ So 
$S^{-1}(h)\cdot f=0\ \forall f$ and $S^{-1}(h)\in I_*.$ Thus $I\subset S(I_*).$     
\end{proof}

So by fixing a basis of a simple self-dual $H$-module $V,$ we can consider the antimorphism induced by the antipode on $End(V)\simeq H/I.$ 

\begin{defn}
Given a self-dual simple $H$-module $V,$ we define $S_V:End(V)\to End(V)$ to be the map induced by $S$ on  $End(V).$  Explicitly, given an endomorphism $\alpha,\ \exists h\in H$ such that 
$\varphi(h)=\alpha.$ Then  $S_V(\alpha)=\varphi (S(h))\in End(V).$ 
\end{defn} 

$S_V$ is the unique map so that the following diagram commutes:

$$
\xymatrix{
H \ar[d] _S \ar[r] ^-\varphi & End(V) \ar@{-->}[d] ^{S_V}\\
H \ar[r]^-\varphi & End(V) }
$$

\begin{defn}\label{ind}
Let $H$ be a Hopf algebra over $\mathbb{C} .$ Given a simple finite dimensional $H$-module $V$, we define the trace invariant associated to $V$ by 
$$
\mu(V):= \left\{ \begin{array}{cl}
trace(S_V), & \mbox{if $V$ is self-dual} \\
 0, & \mbox{otherwise.} 
\end{array} \right .
$$ 
\end{defn}

In the following theorem we determine a bilinear form and an endomorphism associated to a self-dual $H$-module $V$ that allow us to compute $\mu(V)$ in a different way. This result is presumably known but we do not know a reference. The proof we give is similar to one shown to us by Montgomery and Schneider. 

\begin{thm}
Let $V$ be a self-dual simple $H$-module. %Consider $\nu(V)$ as defined in \ref{ind} and let $\langle, \ \rangle$ be the bilinear form given by Theorem \ref{MSch}. 
Then there exists  a non-degenerate $H$-invariant bilinear form $\langle \ ,\ \rangle: V\otimes V\to \mathbb{C}$ and a map $u\in End(V)$ with $\langle v,w\rangle =\langle w,u(v)\rangle$ for all $v,\ w \in V$ such that $ \mu (V)=Tr(u).$
\end{thm}

\begin{proof}
First, let $\phi:End(V)\to M_n(\mathbb{C})$ be the isomorphism mapping an endomorphism to its matrix in some fixed basis. We denote the map $\phi S_V\phi^{-1} : M_n(\mathbb{C})\to M_n(\mathbb{C})$ by $S_V$ also. By the Skolem-Noether theorem, $\exists Q\in GL_n(\mathbb{C})$ such that $\displaystyle S_V(X)=QX^TQ^{-1}\ \forall X\in M_n(\mathbb{C}),$ where $X^T$ is the transpose of $X.$ If we let $U=Q(Q^{-1})^T$ then $S_V^2(X)=UX^TU^{-1},\ \forall X\in M_n(\mathbb{C}).$
Consider the bilinear form given by $$\langle \ , \ \rangle: V \times V \to \mathbb{C}:\ \langle v,w\rangle=v^TQ^{-1}w.$$ 
Then 
\begin{eqnarray*}
\langle v,w\rangle & = & v^TQ^{-1}w \in \mathbb{C}\\
                                 & = &w^T(Q^{-1})^Tv\\
                                 & = &w^TQ^{-1}Q(Q^{-1})^Tv\\
                                 & = &w^TQ^{-1}Uv\\
                                 & = &\langle w,Uv\rangle. 
\end{eqnarray*}

Thus $u=\phi^{-1}(U)\in End(V)$ is the required map. $Tr(u)=Tr(S_V)$ since $S_V(X)=QX^TQ^{-1},$ which implies $Tr(S_V)=Tr(Q(Q^{-1})^T)= Tr(U).$ 

The only thing left to prove is that the bilinear form $\langle \ ,\ \rangle: V\times V \to \mathbb{C}$ is $H$-invariant. 

Let $h\in H$ and $v,w\in V$ and $\varphi:H\to M_n(\mathbb{C})$ be the representation corresponding to $V.$ Then 
\begin{eqnarray*}
h\cdot \langle v,w\rangle & = & \langle h_1\cdot v, h_2\cdot w\rangle\\
                                             & = &\langle \varphi(h_1)(v),\varphi(h_2)(w)\rangle\\
                                             & = &\langle v, S_V(\varphi(h_1))\circ \varphi(h_2)(w)\rangle\\
                                             & = &\langle v,\varphi (S(h_1)h_2)(w)\rangle\\
                                             & = &\epsilon(h)\langle v,w\rangle. 
\end{eqnarray*}  

So the form is $H$-invariant. 
\end{proof}

Each self-dual irreducible representation of $H$ admits a non-degenerate, bilinear form that is $H$-invariant. In general the form is not necessarily symmetric or skew-symmetric, instead we have the map $u$ whose eigenvalues determine the value of the trace $\mu(V).$ In the case of pivotal Hopf algebras, $S^2$ is inner via an element $g\in G(H)$ such that $S^2(h)=ghg^{-1} \ \forall h\in H.$ We could consider $\varphi(g)$ instead of the endomorphism $u.$ However the above theorem guarantees that $tr(u)=tr(S_V)$ and $u$ arises from the bilinear form, that is why we utilize this specific map $u.$  

Now suppose that $\Phi: H'\to H$ is a surjective Hopf algebra homomorphism and let $V$ be a simple $H$-module with corresponding representation $\varphi: H\to End(V).$ Then 
 $$\varphi' :=\varphi\Phi: H'\to End(V)$$ 
 determines an irreducible representation of $H'.$ We will denote the $H'$-module structure of the vector space $V$ by $V'$ when necessary to differentiate it from the $H$-module. We get the following commutative diagram:
$$
\xymatrix{
H' \ar[d] _{S'} \ar[r] ^\Phi \ar@/^2pc/[rr] ^{\varphi '} & H \ar[d] _S \ar[r] ^-\varphi & End(V) \ar[d] ^{S_V=S_{V'}}\\
H' \ar[r] ^\Phi \ar@/_2pc/ [rr] _{\varphi '} & H \ar[r]^-\varphi & End(V) }
$$

By construction,  $\mu (V)=\mu (V')$ and we can compute $\mu$ in the most convenient case and have the information for both the $H$ and the $H'$-module. We want to determine the indicators for the irreducible representations of the Drinfeld double of the Taft algebras $D(H_n(q^2)).$ We do this by considering the projection of $D(H_n(q^{2}))$ onto $u_q(sl_2).$ We will determine the invariant for the  $D(H_n(q^{2}))$-modules by identifying them with $u_q(sl_2)$-modules and computing the invariant for these last ones. 

%---------------------Section-------------------------

\section{The Hopf algebra $u_q(sl_2)$}

We start by  computing the indicators of the irreducible representations of  $u_q(sl_2),$ where $q$ is a primitive $nth$-root of $1,$ following the description given in \cite[p 134-138]{K}. 

First we introduce  some notation for the $q$-binomial coefficients that will be needed later, for $q$ as above. For $i\in \mathbb{Z},$ set 
$$[i]=\frac{q^i-q^{-i}}{q-q^{-1}}.$$
Note that $[-i]=-[i].$ We have the following versions of factorial and binomial coefficients \cite[p 122]{K}: 
$$[0]!=1,\ [k]!=[1][2]\cdots [k]\ \mbox{for}\  k>0$$
and 
$$\left[ \begin{array}{c} 
i\\
k\\
\end{array} \right] = \frac{[i]!}{[k]![i-k]!},\ \mbox{for}\ 0\leq k
\leq i.$$

Consider the algebra $U_q=U_q(sl_2),$ generated by the variables $E,F,K,K^{-1}$ with relations
$$KK^{-1}=K^{-1}K=1,$$
$$KEK^{-1}=q^2E,\ KFK^{-1}=q^{-2}F,$$
and
$$[E,F]=\frac{K-K{-1}}{q-q^{-1}}.$$

The Hopf algebra structure of $U_q(sl_2)$ is determined by
$$\Delta (E)=1\otimes E+E\otimes K, \ \ \Delta (F)=K^{-1}\otimes F+F\otimes 1,$$
$$\Delta (K)=K\otimes K, \ \ \Delta (K^{-1})=K^{-1}\otimes K^{-1},$$
$$\epsilon (E)=\epsilon (F)=0,\ \ \epsilon (K)=\epsilon (K{-1})=1,$$
and
$$S(E)=-EK^{-1},\ S(F)=-KF,\ S(K)=K^{-1},\ S(K^{-1})=K.$$

Let $v$ be a highest weight vector of weight $q^i.$ %$\epsilon q^i,\ \epsilon=\pm 1.$  
Consider the vectors $\displaystyle v_o=v,\ v_j=\frac{1}{[j]!}F^j(v),\ 0<j<i.$ Then $V_i$ is a $U_q$-module of dimension $i+1$ generated by  $v.$ The action of  $U_q$ on the basis $\{v_j\}$ is given by
$$Kv_j=q^{i-2j}v_j,$$
$$Ev_j=[n-p+1]v_{j-1},$$
and
$$Fv_{j-1}=[j]v_j.$$

We will consider an $n^3$-dimensional quotient of $U_q$:

\begin{defn}
The algebra $u_q=u_q(sl_2)$ is the quotient of $U_q$ by the two sided ideal $<E^n, F^n,K^n-1>.$
\end{defn}

For the algebra $u_q,\ S^2:u_q\to u_q$ is an inner automorphism given by conjugation by $K$. In particular, since $K$ is a group-like element, $u_q$ is pivotal. 

The following theorem gives a comprehensive list of simple $u_q$-modules, in the cases we will study.

\begin{thm}\cite[Theorem VI.5.7]{K}
Let $n$ be odd. Any non-zero simple finite dimensional $u_q$- module is isomorphic to a module of the form $V_i$ with $0\leq i< n-1,$ or $V(q^{-1}),$ where $V(q^{-1})$ is an $n$-dimensional vector space, with module structure 
$$Kv_j=q^{-1-2j}v_j,$$  
$$Ev_{j+1}=\frac{q^{-j-1}-q^{j+1}}{q-q^{-1}}[j+1]v_j,$$
$$Fv_j=v_{j+1},$$
 $$Ev_0=0,\ Fv_{n-1}=0,\ \mbox{and} \ Kv_{n-1}=q^{1-2n}v_{n-1}.$$
\end{thm}

So there is exactly one simple module of dimension $i,$ for each $1\leq i \leq n.$ In particular, each of them is self-dual. 

%Consider $E,\ F$ and $K$, the generators of $u_q$ as an algebra.  
%$$S^2(E)=S(-EK^{-1})=-S(K^{-1})S(E)=KEK^{-1},$$
%$$S^2(F)=S(-KF)=-S(F)S(K)=KFK^{-1}$$
%and  
%$$S^2(K)=S(K^{-1})=K=KKK^{-1}$$
%where $E,F$ and $K$ generate $u_q$ as an algebra. 

 For a simple module $V,$ fix a  basis $\mathcal{B}$ of $V$ and let $\varphi:u_q \to M_n(\mathbb{C})$ be the corresponding representation. Then the induced map $S_V^2: M_n(\mathbb{C})\to  M_n(\mathbb{C})$ is conjugation by $\varphi (K).$ % From the proof of Theorem \ref{MSch}, 
$\exists Q\in M_n(\mathbb{C})$ so that $S_V(X)=QX^TQ^{-1},$ and if $U:=Q(Q^{-1})^T,$ we have that $S_V^2(X)=UXU^{-1}$ for all $X.$ In particular, $U\cdot \varphi (K)^{-1}=\lambda I,\ \lambda \in \mathbb{C}$ and we have %The Theorem also states that 
$Tr(S_V)=Tr(U).$ We will use $\varphi (K),$ which is known,  to find $Tr(S_V).$   

Let $\varphi _i$ be the $(i+1)-$dimensional representation of $u_q.$ Then 
$$\varphi_i(K)= \left( \begin{array}{ccccc}
q^i & 0 & \cdot &\cdot & 0\\
0 & q^{i-2} & \cdot &\cdot & 0\\
\cdot & & \cdot & & \cdot \\
\cdot & & & \cdot & \cdot \\
0 & \cdot & \cdot & \cdot & q^{-i}\\ 
\end{array} \right ).
$$ 

Thus $$det(\varphi_i(K))= \prod_{j=0}^i q^{i-2j}=q^{\sum _j i-2j}=q^{i(i+1)}q^{-2\sum j}=q^{i(i+1)-i(i+1)}=1,\ \forall i.$$

As for the trace, $$Tr(\varphi_i(K))=\sum_{j=0}^i q^{i-2j}=q^i \sum_{j=0}^i (q^{-2})^j=[i+1].$$

For each $V_i,\ Tr(\varphi_i(K))$ agrees with the quantum dimension $dim_q(V_i),$ as described in \cite[p 364]{K}. We will show below that the trace invariant $\mu(V_i)$ agrees with these up to sign.  

Let  $Q$ be such that $S_{V_i}(X)=QX^TQ^{-1}$ and write $Q=(Q_{rs})\in M_l(\mathbb{C})$ and $l=i+1.$ Then  $S_{V_i}(X)Q=QX^T$ for all $X.$ Substituting $X=\varphi_i(E)$ and $X=\varphi_i(F)$ in the equation we find that

$$Q_{r,l+1-r}=(-1)^{r-1}\left[ \begin{array}{c} 
l-1\\
r-1\\
\end{array} \right]^{-1} q^{(r-1)i-\sum_{j=1}^{r-1}2j}
$$
and $$Q_{rs}=0 \mbox{ if } r+s\neq l+1.$$

Then, since $\left[ \begin{array}{c} 
l-1\\
r-1\\
\end{array} \right]
=
\left[ \begin{array}{c} 
l-1\\
l-r\\
\end{array} \right]$ for each $r,\ \displaystyle \frac{Q_{r,l+1-r}}{Q_{l+1-r,r}}=(-q)^{l-2r+1}$. 
\vspace{0.1in}

Now given that $U=Q(Q^{-1})^T,$ 
$$Tr(U)=\sum_{r=1}^l \frac{Q_{r,l+1-r}}{Q_{l+1-r,r}}=\sum_{r=1}^l (-q)^{l-2r+1}=
\left\{ \begin{array}{cl}
Tr(\varphi_i(K)),\ & \mbox{if}\ i=l-1\ \mbox{is even}\\
-Tr(\varphi_i(K)),\ & \mbox{if}\ i=l-1\ \mbox{is odd.}
\end{array}\right .
$$
\begin{thm}
Let $V_i$ be the $(i+1)$-dimensional simple module of $u_q,$ as described above. 
Then $$\mu (V_i)=\sum_{j=0}^i (-q)^{i-2j}.$$ %CAMBIE UN 1 POR -1 en el exponente
\end{thm}

\begin{cor}
For $0<i<n-1,\ \mu(V_i)\neq 0$ and $\displaystyle \mu(V_i)=(-1)^{i}dim_q(V_i).$ For the $n$-dimensional simple module $V_{n-1}=V(q^{-1}),\ \displaystyle \mu(V(q^{-1}))=dim_q(V_{n-1})=0.$
\end{cor}

\begin{proof}
In each case, the formula given in the theorem turns out to be a sum of at most $n$ powers of $q$. Since $q$ is a primitive $n$th root of unity, the sum is an algebraic integer for all $i$ and is $0$ when we add exactly $n$ different powers of $q$. 
\end{proof}
\bigskip

%--------------------------Section-----------------------------

\section{The Drinfeld double of the Taft algebra}
Next we introduce the irreducible representations of $D(H_n(q))$ by looking at an isomorphic Hopf algebra $H_n(1,q)$, following \cite {Ch1} and \cite{Ch2}.

Given a primitive $n$th root of unity $q\in k$, we use the description of the irreducible  $H_n(1,q)$-modules given in \cite{Ch2} to study which of those modules are self-dual. Since $D(H_n(q))\cong H_n(1,q),$ that determines 
how many of the simple modules of the Drinfeld Double of a Taft algebra are self-dual.\\

In \cite{Ch2}, Chen defines the Hopf algebra $H_n(p,q)$ where $n\in \mathbb{N} \ p,q\in k.$ We will consider the case when $p=1$ and $q$ is a primitive $n$th root of unity. 
$H_n(1,q)$ is an $n^4$-dimensional algebra generated by $a,b,c,d$ with relations

\begin{eqnarray*}
a^n=0,\ b^n=1,& c^n=1,\ d^n=0,\\
ba=qab,& db=qbd,\\
ca=qac,& dc=qcd,\\
bc=cb,& da-qad=1-bc.
\end{eqnarray*}

and antipode
$$S(a)=-ab^{-1},\ S(b)=b^{-1} ,\ S(c)=c^{-1} \ \mbox{and}\ S(d)=-dc^{-1} .$$
\indent To give the characterization of the $H_n(1,q)$-modules, we need to introduce some notation. 
For any  $i\in \mathbb{N},$ set $$(i)_q=1+q+\cdot \cdot \cdot +q^{i-1}=\frac{q^i-1}{q-1}$$ and 
$$(i)!_q=\frac{(q^i-1)(q^{i-1}-1)\cdot \cdot \cdot (q-1)}{(q-1)^i}.$$  
%%%%%%%%%
As noted in \cite[p 122]{K}, these are related to their $q$-analogues defined in the previous section via 

$$(i)_{q^2}=q^{i-1}[i],\ \text{and}\ (i)!_{q^2}=q^{\frac{i(i-1)}{2}}[i]!$$ \\
\indent Consider the canonical basis $\{x^ig^j:\ 0\leq i,j<n\}$ of the Taft algebra $H_n(q)$ and let $\{\overline{x^ig^j} \}$ be the dual basis of $H_n(q)^*.$  $H_n(1,q)$ is isomorphic to $D(H_n(q))$ via 
$$\overline{x^sg^t} \bowtie x^ig^j \to \sum_{0\leq m<n} \frac{1}{n(s)!_q}q^{-tm}c^md^sa^ib^j.$$

For any $l\in \mathbb{N}$ with $1\leq l\leq n,$ let $$\alpha_i(l)=(i)_q(1-q^{i-l}),\ 1\leq i\leq l-1.$$

Notice that $\alpha_1(l)=1-q^{1-l}$ and $q\alpha_{l-1}(l)=q^{l-1}-1.$ In general, 
$$\alpha_i(l)=1-q^{2i-1-l}+q\alpha_{i-1}(l),\ 1\leq i\leq l-1.$$
 
Let $1\leq l\leq n,\ r\in \mathbb{Z}.$ Then $V(l,r)$ is an $l$-dimensional simple $H_n(1,q)$-module, with action of the generators $a,\ b,\ c,\ d$ on some basis 
$\{v_1,\ v_2, ..., v_l\}$ of $V$ given by 

\begin{eqnarray*}
a\cdot v_i=v_{i+1},&1\leq i<l,&a\cdot v_l=0,\\
b\cdot v_i=q^{r+i-1}v_i,&c\cdot v_i= q^{i-(r+l)}v_i,&1\leq i\leq l,\\
d\cdot v_i=\alpha _{i-1}(l)v_{i-1},&1<i\leq l,&d\cdot v_1=0. \\
\end{eqnarray*}

\begin{thm} \cite{Ch2} \label{Chen}
Let $1\leq l,l'\leq n,\ r,r'\in \mathbb{Z}.$ Then $V(l,r)$ is isomorphic to $V(l',r')$ if and only if $l=l'$ and $r\equiv r' (mod\ n).$
\end{thm}

Thus $V(l,r),\ 1\leq l,r\leq n$ is a complete set of representatives of the $n^2$ isomorphism classes of irreducible representations of $H_n(1,q).$
Given $V=V(l,r),$ we want to determine $l'$ and $s\in \mathbb{Z}$ such that $V^*\cong V(l',s).$ 

Since $dim(V)=dim(V^*),$ we know that $l=l'.$ To find $s$, we study the action of $H_n(1,q)$ on $V^*.$ 

First we need the following lemma.

\begin{lemma} \label{alpha}
Let $1< i\leq l-1.$ Then $$\alpha_{l-(i-1)}(l)q^{2(i-1)-l}=\alpha_{i-1}(l)$$
\end{lemma}

\begin{proof}
First, let $i=2$. 
$$\alpha_{l-1}(l)q^{2-l}=q\alpha_{l-1}(l) q^{1-l}=(q^{l-1}-1)q^{1-l}=1-q^{1-l}=\alpha _1 (l).$$

Now suppose $2<i<l-1.$ We show the condition holds for $i+1$. 
\begin{eqnarray*}
\alpha_{l-i}(l)q^{2i-l}&=&(\alpha_{l-(i-1)}(l)-1+q^{2(l-(i-1))-1-l})q^{2i-l-1}\\
                                    &=&q\alpha_{i-1}(l)-q^{2i-l-1}+1\\
                                    &=&\alpha_i(l)
\end{eqnarray*}
\end{proof}

Recall that a Hopf algebra $H$ acts on the dual of an $H$-module $V$ via the antipode $S$. That is $(h\cdot f)(v)=f(S(h)\cdot v).$ We use this rule to compute the action on  $\{e_1,..., e_l\},$ the basis of $V^*$ dual to $\{v_1,\ v_2, ..., v_l\}.$ 
\begin{eqnarray*}
(a\cdot e_i)(v_j)&=&e_i(S(a)\cdot v_j)=-e_i(ab^{-1}\cdot v_j)\\
                          &=&-q^{-(r+j-1)}e_i(a\cdot v_j)=-q^{-(r+j-1)}e_i(v_{j+1})\\
                          &=&-q^{-(r+j-1)}\delta _{i,j+1},\ if \ 1\leq j<l.\\
\end{eqnarray*}
When $j=l,\ a\cdot v_j=0$ and $(a\cdot e_i)(v_j)=0\ \forall i.$ 
%Thus $$a\cdot e_i=-q^{-(r+i-2)}e_{i-1},\ 1<i\leq l, \ a\cdot e_1=0.$$ 

$$(b\cdot e_i)(v_j)=e_i(S(b)\cdot v_j)=e_i(b^{-1}\cdot v_j)=q^{-(r+j-1)}e_i(v_j) =q^{-(r+i-1)}\delta _{i,j} ,$$

$$(c\cdot e_i)(v_j)=e_i(S(c)\cdot v_j)=e_i(c^{-1}\cdot v_j)=q^{r+l-j}e_i(v_j)=q^{r+l-i}\delta _{i,j}.$$
                          
%Thus $$b\cdot e_i=q^{1-(r+i)}e_i,\ c\cdot e_i=q^{r+l-i}e_i,\ 1\leq i\leq l.$$ 

\begin{eqnarray*}
(d\cdot e_i)(v_j)&=&e_i(S(d)\cdot v_j)=-e_i(dc^{-1}\cdot v_j)\\
                         &=&-q^{r+l-j}e_i(d\cdot v_j)=-q^{r+l-j}\alpha _{j-1}(l)e_i(v_{j-1})\\
                         &=&-q^{r+l-(i+1)}\alpha _i(l)\delta _{i,j-1},\ if \ 1\leq i<l.\\
\end{eqnarray*}                         

When $j=l,\ d\cdot v_j=0$ and $(d\cdot e_i)(v_j)=0\ \forall i.$
%Thus $$d\cdot e_i=-q^{r+l-(i+1)}\alpha _i(l)e_i,\ 1\leq i<l,\ d\cdot e_l=0.$$

Thus 
\begin{eqnarray*}
a\cdot e_i=-q^{-(r+i-2)}e_{i-1},&1< i\leq l,&a\cdot e_1=0,\\
b\cdot e_i=q^{1-(r+i)}e_i,&c\cdot e_i= q^{r+l-i}e_i,&1\leq i\leq l,\\
d\cdot e_i=-q^{r+l-(i+1)}\alpha _i(l)e_i,&1\leq i< l,&d\cdot e_l=0. \\
\end{eqnarray*}

Comparing these results with the action on the modules $V(l,r)$, we see that each $e_i$ is a common eigenvector of $b$ and $c$, as required. However for the action of $a$ and $d$ to be consistent, we need to reverse the order of the basis elements. 

\begin{thm}
If $V=V(l,r)$, then $V^*\cong V(l,1-(r+l)).$ Moreover, if we let $s=1-(r+l)$ and consider the basis  
$$\{ g_i=(-1)^{i-1}q^{-s(l-i)+\frac {i(i-1)}{2}}f_i,\ 1\leq i\leq l\}$$
of $V^*,$ where  $f_i=e_{l-(i-1)},$ then 
  
\begin{eqnarray*}
a\cdot g_i=g_{i+1},&1\leq i<l,&a\cdot g_l=0,\\
b\cdot g_i=q^{s+i-1}g_i,&c\cdot v_i= q^{i-(s+l)}g_i,&1\leq i\leq l,\\
d\cdot g_i=\alpha _{i-1}(l)g_{i-1},&1<i\leq l,&d\cdot g_1=0. \\
\end{eqnarray*} 
  
%In particular, $V^*\cong V(l,s).$
\end{thm}

\begin{proof}
Since $f_i=e_{l-(i-1)},$ then $f_1=e_l,\ f_2=e_{l-1}, ..., f_l=e_1$ and the action is given by

\begin{eqnarray*}
a\cdot f_i=-q^{i+1-(r+l)}f_{i+1},&1\leq i<l,&a\cdot f_l=0,\\
b\cdot f_i=q^{i-(r+l)}f_i,&c\cdot v_i= q^{i+r-1}f_i,&1\leq i\leq l,\\
d\cdot f_i=-\alpha _{l-(i-1)}(l)q^{r+i-2}f_{i-1},&1<i\leq l,&d\cdot f_1=0. \\
\end{eqnarray*} 

Because the actions on the basis $\{f_i\}$ agree with the ones describing $V(l,1-(r+l))$ up to some constants, we can conclude that $s=1-(r+l).$ Since each  $f_i$ is an eigenvector for both $b$ and $c,$ so is each $g_i$ and it corresponds to the same eigenvalue. We still have to check the actions of $a$ and $d$ on this new basis.

Let $1\leq i<l.$ Then
\begin{eqnarray*}
a\cdot g_i&=&(-1)^{i-1}q^{-s(l-i)+\frac{i(i-1)}{2}}a\cdot f_i\\
                  &=& (-1)^{i-1}q^{-s(l-i)+\frac{i(i-1)}{2}}(-q^{i+s})f_{i+1}\\
                  &=&(-1)^iq^{-s(l-(i+1))+i+\frac{i(i-1)}{2}}f_{i+1}\\
                  &=&g_{i+1}.\\
\end{eqnarray*} 

When $i=l$ we have

\begin{eqnarray*}
a\cdot g_l&=&((-1)^{l-1}q^{\frac{l(l-1)}{2}})a\cdot f_l=0.\\
\end{eqnarray*} 

Now let $1<i\leq l.$ Then
\begin{eqnarray*}
d\cdot g_i&=&(-1)^{i-1}q^{-s(l-i)+\frac{i(i-1)}{2}}d\cdot f_i\\
                  &=& (-1)^{i-1}q^{-s(l-i)+\frac{i(i-1)}{2}}(-\alpha_{l-(i-1)}(l))q^{-s-(l-(i-1))}f_{i-1}\\
                  &=&(-1)^{i-2}q^{-s(l-(i-1))+\frac{(i-1)(i-2)}{2}}q^{2(i-1)-l}\alpha _{l-(i-1)}(l)f_{i-1}\\
                  &=&q^{2(i-1)-l}\alpha _{l-(i-1)}(l)g_{i-1}\\
                  &=&\alpha_{i-1}g_{i-1},\\
\end{eqnarray*} 
where the last equality follows from \ref{alpha}.

When $i=1$ we have 
 \begin{eqnarray*}
 d\cdot g_1&=&q^{-s(l-1)}d\cdot f_1=0.
 \end{eqnarray*}

\end{proof}

We can now determine which of the $H_n(1,q)$-modules are self-dual.

\begin{thm}
Let $n\in \mathbb{N}$ and $1\leq l,r \leq n$.  $V(l,r)$ is a self-dual $H_n(1,q)$-module only in the following cases:
\begin{enumerate}
\item
When $n$ is odd, there is exactly one $l$-dimensional self-dual  module $V(l,r)$ for each $l$. If $l$ is odd, $r=n+\frac{1-l}{2}$. If $l$ is even, $r=\frac{n+1-l}{2}.$
\item
When $n$ is even, there are two self-dual modules $V(l,r)$ for each odd $l$, given by $r=\frac{n+1-l}{2}$ and $r=n+\frac{1-l}{2}.$
\end{enumerate} 
\end{thm}

\begin{proof}
Fix $n\in \mathbb{N}$. We know that $V(l,r)^*\cong V(l,1-(r+l))$. Then by \ref{Chen}, $V(l,r)$ is self-dual if and only if $r\equiv_n 1-(r+l).$ Equivalently, if and only if $n|\ 2r+l-1.$ If we consider $1\leq r,l\leq n$ then $2\leq 2r+l-1\leq 3n-1.$ Thus the only cases to consider are $2r+l-1=n$ and $2r+l-1=2n.$   

\begin{enumerate}
\item
If, given $l$,  $2r+l-1=n$ then $r=\frac{n+1-l}{2}\in \mathbb{N}$ if and only if $n$ and $l$ have different parities.
\item
If, given $l$, $2r+l-1=2n$ then $r=\frac{2n+1-l}{2}=n+\frac{1-l}{2} \in \mathbb{N}$ if and only if $l$ is odd.
\end{enumerate}
$V(l,r)$ is self-dual only for the pairs $(l,r)$ above.
\end{proof}

We intend to use the information we have already computed for the representations of $u_q(sl_2)$ to determine the indicators for simple modules of the double. When $n$ is odd, $q$ is a primitive $nth$ root of $1$ if and only if $q^2$ is. We can project $H_n(1,q^{2})$ onto $u_q(sl_2)$ and through this map we obtain the indicators for the simple $H_n(1,q^{2})$-modules.

\begin{prop}
The map $\psi: H_n(1,q^2) \to u_q(sl_2)$ given by $\psi(a)=E,\ \psi(b)=K,\ \psi (c)=K$ and $\psi(d)=\displaystyle \frac{q-q^{-1}}{q^2}FK$ is a surjective Hopf algebra morphism.  
\end{prop}

\begin{proof}
We obtain $\psi$ as a composition of Hopf algebra maps, using a projection of a Hopf algebra isomorphic to $H_n(1,q^2)$ onto $u_q(sl_2),$ as given in \cite[p 227-228]{K}. 
\end{proof}

Every simple $u_q(sl_2)$-module is thus an $H_n(1,q^2)$ simple module. But then each of them has to agree with one of the $V(l,r),$ as described above. 

\begin{thm}
When looked at as an $H_n(1,q^2)$-module via $\psi,$ the simple (and self-dual) $u_q(sl_2)$-module $V_i$ corresponds to $V(i+1,n-\frac{i}{2})$ if $i$ is even, and $V(i+1,\frac{n-i}{2})$ if $i$ is odd. $V(q^{-1})$ corresponds to the $n$-dimensional module $V(n,\frac{n+1}{2}).$
\end{thm}

\begin{proof}
Consider the basis $\{v_p:\ 0\leq p\leq i\}$ of $V_{1,i}.$ Then
$$a\cdot v_p=\psi(a)v_p=Ev_p=[i-p+1]v_{p-1},$$
$$b\cdot v_p=\psi(b)v_p=Kv_p=q^{i-2p}v_p,$$
$$c\cdot v_p=\psi(c)v_p=Kv_p=q^{i-2p}v_p$$
and
$$d\cdot v_p=\psi(d)v_p=\frac{q-q^{-1}}{q^2}FKv_p=\frac{q-q^{-1}}{q^2}q^{i-2p}[p+1]v_{p+1}.$$

The actions of $a$ and $d$ suggest that the order of the vectors in the basis should be changed. If we reverse the basis and consider $\{w_p=v_{i-p}\},$ the action of $b$ and $c$ on this basis is given by the matrix 
 
$$\left( \begin{array}{ccccc}
q^{-i} & 0 & \cdot & \cdot & 0\\
0 & q^{2-i} & \cdot & \cdot & \cdot\\
\cdot &  & \cdot & & \cdot\\ 
\cdot &  &  & \cdot & 0\\
0 & \cdot & \cdot & \cdot & q^i\\                 
\end{array}\right) $$

which agrees with the action of $b$ and $c$ on $V(i+1, r)$ where $r$ is so that $2r\equiv _n -i:$
in that case $q^{2p-i}=q^{2p+2r}=(q^2)^{(p+r)},$ and changing the indices to $j=1,\cdots i+1$ as in the definition of $V(l,r)$ the $j$th entry in the diagonal becomes $(q^2)^{j+r-1}.$ This happens to be the unique self-dual $i+1$-dimensional representation of $H_n(1,q^2),$ as it should be. 
\end{proof}

We can now determine our original goal: compute the invariant for the representations of the double of the Taft algebra, combining all the preceding results.%using the projection of $D(H_n(q^2))$ onto $u_q(sl_2)$combined with the values of $\mu$ computed in the previous section.

\begin{thm}
Let $q$ be a primitive $nth$ root of unity, and consider the Hopf algebra $D(H_n(q^2))\cong H_n(1,q^2),$ as described above. The trace associated to the $l$-dimensional module $V(l,r)$ is given by 

$$
\mu(V(l,r))=
\left\{ \begin{array}{cl}
\displaystyle \sum_{j=1}^l (-q)^{l-2j+1},\ &  \mbox{if}\ 2r\equiv_n 1-l,\\
0 & \mbox{otherwise.}
\end{array}\right .
$$

\end{thm}

%%%%%% Section: acknowledgment %%%%%%%%%%
\section*{Acknowledgments}
Most of this work appeared in the author's PhD dissertation at the University of Southern California. I am grateful to the Mathematics Department at USC and in particular my advisor Susan Montgomery for their support.
I would also like to thank Siu-Hung Ng for suggesting the connection of the trace invariant to quantum dimension.

\end{document}